\documentclass{amsart}
\usepackage{epsfig}
\usepackage{graphicx}
\usepackage{amsthm}
\usepackage{amsfonts}
\usepackage{amssymb}
\usepackage[all]{xy}

\newcommand{\Ber}{\mathrm{Ber}}
\newcommand{\ev}{\mathrm{ev}}
\newcommand{\odd}{\mathrm{odd}}

\newtheorem{thm}{Thm}[section]
\newtheorem{satz}[thm]{Theorem}
\newtheorem{lemma}[thm]{Lemma}

\newtheorem{prop}[thm]{Proposition}

\theoremstyle{definition}
\newtheorem{def1}[thm]{Definition}

\theoremstyle{remark}
\newtheorem{rem}[thm]{Remark}

\begin{document}
\title{Measurability of open orbits in flag supermanifolds}
\author{Christopher Graw}
\thanks{Partially supported by the SFB TR-12 of the Deutsche Forschungsgemeinschaft.}
\address{Fakult\"{a}t f\"{u}r Mathematik \\ Ruhr-Universit\"{a}t Bochum \\ D-44780 Bochum}
\email{christopher.graw@rub.de}
\begin{abstract}
Let $G_\mathbb{R}$ be a Lie group and $G$ its complexification. An open $G_\mathbb{R}$-orbit in a 
$G$-flag manifold is measurable whenever it carries a $G_\mathbb{R}$-invariant volume element. In this paper the notion
of measurability is generalized to the supersymmetric setting and a classification of measurable open orbits 
in flag supermanifolds is given. 
\end{abstract}
\maketitle

\section{Introduction}

The orbit structure of the action of a real semisimple group $G_\mathbb{R}$ on a complex manifold $Z$ was first analyzed
 by J.A. Wolf in \cite{W} where for example it was proved that if $Z = G/P$ is a $G$-flag manifold of the complexification
of $G_\mathbb{R}$, then $G_\mathbb{R}$ has finitely many orbits on $Z$. In particular $G_\mathbb{R}$ has open orbits $D$ in $Z$
called flag domains. 

Of particular importance among these flag domains are the measurable ones, that is those carrying
a $G_\mathbb{R}$-invariant volume element. These open orbits allow natural exhaustion functions yielding an important
vanishing theorem for the cohomology of $G_\mathbb{R}$-homogeneous vector bundles on $D$(\cite{SW}). 
Moreover, measurability is used to create unitary structures on representation spaces of $G_\mathbb{R}$. 
A characterization of measurable open orbits in both root-theoretic and geometric terms is given in \cite{W}.
In particular, it is shown that measurability is a property of the ambient flag manifold and not of the open orbit in question.

In the present paper the notion of measurable open orbits is extended to the category of supermanifolds. It turns out that the 
characterization theorem from \cite{W} does not generalize verbatim. This motivates the introduction of two different notions 
of weak and strong measurability. The first of these is the natural generalization of measurability, whereas the second is 
designed to fulfill the requirements of the literal generalization of the Characterization Theorem in \cite{W}.

The main goal of this paper is to prove a detailed classification of measurable open orbits in complex flag supermanifolds.
This includes a list of the measurable open orbits in the classical case which in fact could have been directly derived from
the above-mentioned characterization in \cite{W}.
To this end we use the classification of complex simple Lie superalgebras in \cite{Kac} and their real forms in \cite{Par}, the classication of involutions of simple Lie superalgebras
in \cite{Ser} and results on the structure of parabolic subalgebras due to Onishchik and Ivanova (see \cite{OnI}). 
A table of the results is also given at the end of this text.

\section{Preliminaries}

The main problem adressed in this paper is the existence of an invariant Berezinian density, the natural generalization of an invariant volume element,
on a homogeneous supermanifold. This problem can be stated in terms of the isotropy representation at the neutral point which is a representation of a classical
Lie group on a Lie superalgebra (Theorem 4.13 in \cite{AH}). So the solution of the problem amounts to classical representation theory. 
Nevertheless we recall some basic notions from the theory of supermanifolds.

\subsection{Background on supermanifolds}

This paper uses supermanifolds in the sense of Kostant, Berezin and Leites, i.e. a supermanifold is a ringed space $X = (X_0, \mathcal{O}_X)$, which is locally
isomorphic to $\mathbb{C}^{p \vert q} = (\mathbb{C}^p, \mathcal{O} \otimes \bigwedge (\mathbb{C}^q)^*)$. The complex manifold $X_0$ is called the body of $X$
and $\mathcal{O}_X$ the sheaf of superfunctions on $X$. A morphism of supermanifolds is a morphism
of ringed spaces and the supermanifolds together with these morphisms form a category. Lie supergroups are group objects in the category of supermanifolds,
i.e. they come equipped with multiplication, inversion and unit morphisms satisfying the usual commutativity relations. 

Given a supermanifold $X$, its tangent sheaf is defined as $\mathcal{T}_X = Der \mathcal{O}_X$, the sheaf of derivations of $\mathcal{O}_X$. 
The tangent space $T_{X,x}$ can be identified with $\mathfrak{m}_x/\mathfrak{m}_x^2$, where $\mathfrak{m}_x$ is the unique maximal ideal of the local ring $\mathcal{O}_{X,x}$.
If $G$ is a Lie supergroup, then $\mathfrak{g} := T_{G,e}$ carries the structure of a Lie superalgebra, i.e. a $\mathbb{Z}/2$-graded vector space with a bracket $[\cdot,\cdot]: \mathfrak{g} \oplus \mathfrak{g} \rightarrow \mathfrak{g}$,
which is graded-anticommutative and satisfies the graded Jacoby identity. $\mathfrak{g}$ is called the Lie superalgebra of the super Lie group $G$.

Quotients of Lie supergroups by closed subgroups are defined for example in \cite{Kos} and it was shown there that the tangent space at the neutral point of a quotient $G/P$ is given by $\mathfrak{g}/\mathfrak{p}$.
In this paper $G$ is always a simple complex Lie supergroup and $P$ is a parabolic subgroup. This means that its body $P_0$ is a parabolic subgroup of $G_0$ and $\mathfrak{p}$ is a parabolic subsuperalgebra
of $\mathfrak{g}$ in the sense of \cite{OnI}. As shown in \cite{OnI}, if $\mathfrak{g}$ is one of the classical Lie superalgebras there is a 1-to-1-correspondence between conjugacy classes of parabolic subalgebras
of $\mathfrak{g}$ and flag types of an associated super vector space $V$. For this reason these quotients are called flag supermanifolds. Each flag type is given by a dimension sequence

\[ \delta = 0 \vert 0 < d_0^1 \vert d_1^1 < \ldots < d_0^k \vert d_1^k < n \vert m \]

\noindent and a flag of type $\delta$ is a sequence

\[ 0 \leq V_1 \leq \ldots \leq V_k \leq \mathbb{C}^{n \vert m} \]

\noindent of graded subspaces satisfying $\dim V_i = d_0^i \vert d_1^i$. Flag supermanifolds are defined as follows:
Given a dimension sequence $\delta$, the body $Z(\delta)_0$ is the set of all flags of type $\delta$.
By virtue of the above mentioned 1-1 correspondence there is, for every given flag $z \in Z(\delta)_0$, a unique maximal parabolic
subalgebra $\mathfrak{p} \leq \mathfrak{g}$ stabilizing $z$ and the conjugacy class of $\mathfrak{p}$ in $\mathfrak{g}$
depends only on $\delta$. Therefore we may define $Z(\delta) =  (Z(\delta)_0 , \varphi^* \mathcal{O}_{G/P})$, where
$\varphi : Z(\delta)_0 \rightarrow^\simeq G_0/P_0$ is the canonical isomorphism. The identification of $Z(\delta)$
with $G/P$ will be used implicitly throughout this paper. 

A Berezinian form is a global section of the Berezinian bundle $\Ber(X) = \Ber(T^*X)$. The corresponding sheaf $\mathcal{B}er = \Gamma(\cdot, \Ber(T^*X))$
is isomorphic to the sheafification of the presheaf $\bigwedge^{top} \Omega_{X,0} \otimes S^{top} \mathcal{T}_{X,1}$. Its sections
are given on a coordinate chart $U$ as $ \mathcal{D}(x,\xi) = f dx_1 \wedge \ldots \wedge dx_n \otimes \frac{\partial}{\partial \xi_1} \wedge \ldots \wedge \frac{\partial}{\partial \xi_m}$ for some $f \in \mathcal{O}_X(U)$.
In particularly, we are interested in the existence of invariant Berezinian forms on open $G_\mathbb{R}$-orbits $D \subseteq Z(\delta)$, where $G_\mathbb{R}$ is a real form of $G$.
The open orbit $D$ is called measurable if it allows a $G_\mathbb{R}$-invariant Berezinian form.
The following three symmetry conditions are central to the characterization of measurability in the supersymemtric setting:

\begin{def1}

Let $\delta$ be a dimension sequence. Then:

\begin{enumerate}
 \item $\delta$ is even-symmetric, if $d_0 \vert d_1 \in \delta$, if and only if $(n- d_0) \vert (m - d_1) \in \delta$
 \item $\delta$ is odd-symmetric, if $n=m$ and $d_0 \vert d_1 \in \delta$, if and only if $(n-d_1) \vert (n-d_0) \in \delta$
 \item $\delta$ is $\Pi$-symmetric , if $n = m$ and $d_0 = d_1$ for all $d_0 \vert d_1 \in \delta$.
\end{enumerate}
\end{def1}
 
Note that in the first and second case, as the product ordering on $\{0, \ldots, n\} \times \{0, \ldots, m\}$ is not total, unless $mn = 0$,
not every given dimension sequence can be enlarged to a symmetric dimension sequence. We call $\delta$ symmetrizable if it can be enlarged to
a symmetric dimension sequence.

\begin{rem}
The even and odd symmetry conditions are related to symmetries of the extended Dynkin diagram $A(n,m)$ in the following way:
Recall that the underlying graph of the extended Dynkin diagram is a cycle on $n+m+2$ nodes, two of which correspond to odd roots. Removing these two nodes
one obtains the disjoint union of the Dynkin diagrams $A_n$ and $A_m$. Now assume the nodes are labelled by simple roots and a parabolic subalgebra $\mathfrak{p}$
of $A(n,m)$ is given. A diagram automorphism of $A(n,m)$ must either fix the two odd roots or interchange them. Thus there are three possible automorphisms: a reflection $r_0$
fixing one or two even roots, a reflection $r_1$ fixing the two odd roots and the antipodal map $s$. The last two of these require $m = n$.

Let $J$ be the set of simple roots such that $\mathfrak{g}^{-\alpha} \subseteq \mathfrak{p}$ and let $Z(\delta) = G/P$. Then $\delta$  is even-symmetric if and only if $r_0(J) = J$.
It is odd-symmetric if and only if $r_1(J) = J$ and it is $\Pi$-symmetric, if $s(J) = J$.  
\end{rem}

It turns out that, unlike in the classical case, measurability is not equivalent to $P \cap \tau P$ being complex reductive:

\subsection{Example}

Let $n > 2$, $G = PSL_{n \vert n}(\mathbb{C})$, $G_\mathbb{R} = PSL_{n \vert n}(\mathbb{R})$ and $Z = \mathrm{Gr}_{1 \vert 1}(\mathbb{C}^{n \vert n})$, i.e. $Z = Z(\delta)$ for
$\delta = 0 \vert 0 < 1 \vert 1 < n \vert n$.
Then $D_0$ is the product of two copies of the open $SL_n(\mathbb{R})$-orbit in $\mathbb{P}^{n-1}(\mathbb{C})$, which is not measurable. 
In a suitable basis of $\mathbb{C}^{n \vert n}$, the involution $\tau$ defining the real form acts by $\tau(X) = A\overline{X}A$, where $A$ is the unit antidiagonal matrix. Therefore

\[ \mathfrak{p} \cap \tau\mathfrak{p} = \left\lbrace 
\begin{pmatrix} 
a_1 & v_1^T & 0 & \vline & a_2 & v_2^T & 0 \\ 
0 & X_1 & 0 & \vline & 0 & X_2 & 0 \\
0 & w_1^T & b_1 & \vline & 0 & w_2^T & b_2 \\ \hline
a_3 & v_3^T & 0 & \vline & a_4 & v_4^T & 0 \\ 
0 & X_3 & 0 & \vline & 0 & X_4 & 0 \\
0 & w_3^T & b_3 & \vline & 0 & w_4^T & b_4
\end{pmatrix} \in \mathfrak{g} : \begin{matrix} a_i,b_i \in \mathbb{C} \\ v_i,w_i \in \mathbb{C}^{n-2} \\ X_i \in \mathfrak{gl}_{n-2}(\mathbb{C})  \end{matrix} \right\rbrace \]

Let $\mathfrak{h}$ be the Cartan subalgebra of diagonal matrices. Then $\mathfrak{h}^*$  is generated by the functionals $x_i - x_j, x_i - y_j$ and $y_i - y_j$, where
$x_i$ is the $i^{th}$ diagonal entry of the upper left block and $y_j$ is the $j^{th}$ diagonal entry of the lower right block. So the root system $\Sigma(\mathfrak{g},\mathfrak{h})$
is given by 

\[ \Sigma(\mathfrak{g},\mathfrak{h}) = \{ x_i - x_j \vert 1 \leq i,j \leq n , i \neq j \} \cup  \{ y_i - y_j \vert 1 \leq i,j \leq n , i \neq j \}\]
\[ \cup  \{ \pm x_i \mp y_j \vert 1 \leq i,j \leq n \} \] 

The root spaces contributing to the quotient $\mathfrak{g}/(\mathfrak{p} \cap \tau\mathfrak{p})$
are precisely the ones corresponding to the roots $(x_1 - x_j),(x_1 - y_j),(y_1-x_j),(y_1-y_j) (j \neq 1)$ and
$(x_n - x_j),(x_n - y_j),(y_n-x_j),(y_n-y_j) (j \neq n)$. 
As their graded sum is zero, $\mathfrak{p}$ acts on $\mathfrak{g}/\mathfrak{p}$ and therefore on $(\mathfrak{g}/\mathfrak{p})^*$  with trace zero. Consequently, the
isotropy action of $P_{0\mathbb{R}}$ on  $\Ber(\mathfrak{g}/\mathfrak{p})^*$ is trivial. This implies
the existence of a $G_\mathbb{R}$-invariant Berezinian form on $D$.

To sum it up, $D$ carries a $G_\mathbb{R}$-invariant Berezinian form, even though $P_0 \cap \tau P_0$ is not reductive and $D_0$
is not measurable.

\section{Measurable flag domains}

The fact that there are flag domains in the supercase which allow an invariant Berezinian, but whose real stabiliser subgroups are not reductive,
shows that the characterization theorem from \cite{W} does not generalize verbatim. This fact motivates the introduction of two different notions
of measurability. The first of these is the naive generalization and the other one is a notion designed to fulfill the requirements of a characterization theorem
analogous to the one given in \cite{W}. 

\begin{def1}
 
Let $G$ be a complex reductive Lie supergroup, $G_\mathbb{R}$ a real form of $G$, $Z = G/P$ a flag supermanifold
and $D \subseteq Z$ an open $G_\mathbb{R}$-orbit. Then:

\begin{enumerate}
 \item $D$ is strongly measurable, if $D$ carries a $G_\mathbb{R}$-invariant Berezinian density and $D_0$ is measurable.
 \item $D$ is weakly measurable, if $D$ carries a $G_\mathbb{R}$-invariant Berezinian density and $D_0$ is not measurable.
\end{enumerate}

\end{def1}

Starting with a real supergroup orbit $D \subseteq Z$ such that $D_0$ is open there are three things to check:

\begin{itemize}
 \item Is the odd dimension of $D$ maximal? 
 \item Is $D$ strongly measurable?
 \item Is $D$ weakly measurable?
\end{itemize}

From now on let $G$ be a classical complex reductive Lie supergroup, $G_\mathbb{R} = \mathrm{Fix}(\tau)$ a real form of $G$, $P$ a parabolic subsupergroup of $G$,
$\mathfrak{g}, \mathfrak{g}_\mathbb{R}$ and $\mathfrak{p}$ the respective Lie superalgebras, $\mathfrak{h}$ a Cartan subalgebra of $\mathfrak{g}$,
$\Sigma = \Sigma(\mathfrak{g},\mathfrak{h})$ a root system and $\Phi^r = \{ \alpha \in \Sigma : \mathfrak{g}^\alpha \oplus \mathfrak{g}^{-\alpha} \in \mathfrak{p} \},
 \Phi^n = \{ \alpha \in \Sigma^+ : \mathfrak{g}^{-\alpha} \not\in \mathfrak{p} \}, \Phi^c = \{ \alpha \in \Sigma^- : \mathfrak{g}^\alpha \not\in \mathfrak{p} \}$.

A useful criterion for maximal odd dimension is given by the following codimension formula:

\begin{satz}
 
Let $Z = G/P$ be a flag supermanifold, $D \subseteq Z$ a real group orbit. Then

\[ \mathrm{codim}_Z(D) = \vert \Phi^c \cap \tau\Phi^c \vert \] 

\noindent Especially, $D$ is open, if and only if $\Phi^c \cap \tau \Phi^c$ is empty.

\end{satz}

This codimension formula is a generalization of Theorem 2.12 in \cite{W} and the proof is completely analogous to the one given there.

Strong measurability is characterized by the following theorem which is analogous to the characterization of measurability 
in the classical setting (see Theorem 6.1.(1a);(2a)-(2d) and 6.7. in \cite{W}).  Weak measurability occurs in some exceptional cases 
where certain symmetries of the root system lead to the cancellation of even and odd roots.

\begin{satz}
Let $G$ be a complex reductive Lie supergroup, $G_\mathbb{R}$ a real form of $G$ and $P$ a parabolic subgroup of $G$
and $D \cong G_\mathbb{R}/L_\mathbb{R}$ an open orbit in $G/P$. Then the following are equivalent.

\begin{enumerate}
 \item $D$ is strongly measurable
 \item $\mathfrak{p} \cap \tau\mathfrak{p}$ is reductive
 \item $\mathfrak{p} \cap \tau\mathfrak{p}$ is the reductive part of $\mathfrak{p}$
 \item $\tau \Phi^r = \Phi^r$ and $\tau \Phi^n = \Phi^c$
 \item $\tau\mathfrak{p}$ is $G_0$-conjugate to $\mathfrak{p}^{op}$.
\end{enumerate}
 
\end{satz}

\begin{proof}

Equivalence of 2.-5. follows form the fact that the reductive part of $\mathfrak{p} \cap \tau\mathfrak{p}$ is $\mathfrak{p}^r \cap \tau\mathfrak{p}^r$ and
its nilpotent radical is $\mathfrak{p}^r \cap \tau\mathfrak{p}^n + \mathfrak{p}^n \cap \tau\mathfrak{p}^r$.  It remains to show the equivalence of 1. and 2.

If $D$ is strongly measurable, then the Berezinian bundle $\Ber (G_\mathbb{R}/L_\mathbb{R})$ has a non-zero $G_\mathbb{R}$-invariant global section. 
By virtue of Theorem 4.13 in \cite{AH} this is equivalent to $\Ber((\mathfrak{g}_\mathbb{R}/\mathfrak{l}_\mathbb{R})^*)$ being a trivial $L_\mathbb{R}$-module.
This in turn is equivalent to $\Ber(\mathfrak{g}_\mathbb{R}/\mathfrak{l}_\mathbb{R})$ being a trivial $L_\mathbb{R}$-module.
Consequently the Harish-Chandra pair $\mathfrak{l}_\mathbb{R}$ and its complexification $\mathfrak{p} \cap \tau \mathfrak{p}$
act on $\mathfrak{g}_\mathbb{R}/\mathfrak{l}_\mathbb{R}$ and $\mathfrak{g}/(\mathfrak{p} \cap \tau \mathfrak{p})$ respectively
by operators with vanishing supertrace. As $\mathrm{ad}(\mathfrak{g}_1) \subseteq \mathfrak{gl}(\mathfrak{g})_1$, all of these
operators will have supertrace equal to zero. Hence one only needs to consider the action of $\mathfrak{p}_0 \cap \tau \mathfrak{p}_0$
on $\mathfrak{g}/(\mathfrak{p} \cap \tau \mathfrak{p})$.


By assumption, $D_0$ is measurable and consequently $\mathfrak{p}_0 \cap \tau \mathfrak{p}_0$ is reductive. Therefore  $\mathfrak{p}_0 \cap \tau \mathfrak{p}_0$
acts on $\mathfrak{p}_0 \cap \tau \mathfrak{p}_0$ and on $\mathfrak{p}_1 \cap \tau \mathfrak{p}_1$ with trace zero. 
As in the proof of the analogous theorem
in the classical case, one observes that $\mathfrak{p}_1 \cap \tau \mathfrak{p}_1$ splits into 
$\mathfrak{p}_1^r \cap \tau \mathfrak{p}_1^r \oplus \mathfrak{p}_1^r \cap \mathfrak{p}_1^n  \oplus \mathfrak{p}_1^n \cap \tau \mathfrak{p}_1^r$ 
and the trace zero condition forces $\mathfrak{p}_1^r \cap \mathfrak{p}_1^n  \oplus \mathfrak{p}_1^n \cap \tau \mathfrak{p}_1^r = 0$. 
Consequently, $P \cap \tau P$ is reductive.

Conversely, if $P \cap \tau P$ is reductive, then its real form $L_\mathbb{R}$ is real reductive and therefore acts trivially on $\Ber((\mathfrak{g}_\mathbb{R}/\mathfrak{l}_\mathbb{R})^*)$. 
Then again by virtue of \cite{AH}, there exists a $G_\mathbb{R}$-invariant Berezinian form on $D$. Moreover, $P_0 \cap \tau P_0$ is reductive and therefore $D_0$ is measurable by the classical characterization theorem.

\end{proof}

\begin{rem}

Using the above characterization, non-measurable open orbits can be constructed as follows:
According to \cite{OnI} any parabolic subalgebra $\mathfrak{p}$ of a basic Lie superalgebra $\mathfrak{g}$ is of the form

\[ \mathfrak{p} = \mathfrak{b} \oplus \bigoplus_{ \alpha \in [-J]} \mathfrak{g}^\alpha \]

\noindent where $\mathfrak{b}$ is a Borel subalgebra, $J \subseteq \Pi$ is a set of simple roots and $[-J]$ is the set of all negative roots
which are sums of elements of $-J$. Now, if $\mathfrak{p}$ is given as above, open $G_\mathbb{R}$-orbits in $G/P$ will be measurable
if and only if $\tau(J) = J$. So, in order to construct a non-measurable open orbit, one only needs to pick a subset $J \subseteq \Pi$
such that $\tau(J) \neq J$ and let $\mathfrak{p}$ be given by the above formula.  

\end{rem}

The proof of strong measurability can be simplified by using the following lemma:

\begin{lemma}
 
A flag domain $D$ is strongly measurable if and only if for every $\alpha \in \Sigma: \alpha \in \Phi \Leftrightarrow \tau(-\alpha) \in \Phi$

\end{lemma}

\begin{proof}
 
Let $D$ be strongly measurable. Then $\alpha \in \Phi \cap \tau \Phi$ if and only if $- \alpha \in \Phi \cap \tau \Phi$.
Suppose $\alpha \in \Phi$. If $\tau(\alpha) \in \Phi$, then $-\alpha \in \Phi \cap \tau \Phi$ and therefore $\tau(-\alpha) \in \Phi$.
If $\tau(\alpha) \not\in \Phi$, then $\alpha \in \Phi^n$ and, as $D$ is strongly measurable, $\tau(\Phi^n) = -\Phi^n$, and therefore
$\tau(-\alpha) \in -(-\Phi^n) = \Phi^n$, especially $\tau(-\alpha) \in \Phi$. The same argument applied to $\tau(-\alpha)$ yields the converse.
So if $D$ is strongly measurable, $\alpha \in \Phi$ if and only if $\tau(-\alpha) \in \Phi$.

Now suppose $\alpha \in \Phi \Leftrightarrow \tau(-\alpha) \in \Phi$. If $\alpha \in \tau(\Phi)$, then $\tau(\alpha) \in \Phi$
and therefore $\tau(-\tau(\alpha)) = -\alpha \in \Phi$. As $\tau(-\alpha) \in \Phi$ was given, this implies $-\alpha \in \Phi \cap \tau \Phi$.
If $\alpha \not\in \tau(\Phi)$, then $\tau(\alpha) \not\in \Phi$ and therefore $- \alpha = \tau(-\tau(\alpha)) \not\in \Phi$.
So $\alpha \in \Phi \cap \tau\Phi$ if and only if $-\alpha \in \Phi \cap \tau\Phi$ and therefore $D$ is strongly measurable as claimed.

\end{proof}

Note that a $G_\mathbb{R}$-orbit $D$ is measurable if the super-trace of the natural representation of $\mathfrak{p}$ on $(\mathfrak{g}/\mathfrak{p})^*$
is zero. Therefore strong measurability requires $\alpha \in \Phi \cap \tau \Phi \Leftrightarrow -\alpha \in \Phi \cap \tau \Phi$
while weak measurability requires that the even and odd roots cancel each other. 


\section{Classification of measurable open orbits}

The measurability of open real orbits will be analyzed case by case according to the classification of complex simple Lie superalgebras.
The classification of real forms and part of the notation are taken from \cite{Ser}. The results are summarized in a table at the end of this text.

\subsection{Type A(m,n)}

\begin{satz}
 
Let $\mathfrak{g} = \mathfrak{sl}_{m \vert n}(\mathbb{C})(m \neq n)$ or $\mathfrak{g} = \mathfrak{psl}_{n \vert n}(\mathbb{C})$. 

\begin{enumerate}
 \item If $\mathfrak{g}_\mathbb{R} = \mathfrak{sl}_{m \vert n}(\mathbb{R})$ or $\mathfrak{g}_\mathbb{R} = \mathfrak{sl}_{k \vert l}(\mathbb{H})(n=2k,m=2l)$, then
  \begin{itemize}
   \item A flag domain $D$ is open if and only if $\delta$ is even-symmetrizable.
   \item A flag domain $D$ is strongly measurable if and only if $\delta$ is even-symmetric.
   \item If $m=n$ and $\delta$ is $\Pi$-symmetric then $D$ is weakly measurable.
   \item If $n=m$, the unique open $G_\mathbb{R}$-orbit in $\mathbb{P}(\mathbb{C}^{n \vert n})$ is weakly measurable.
  \end{itemize}
 \item If $\mathfrak{g}_\mathbb{R} = \mathfrak{su}(p, n-p \vert q, m-q)$ then $D$ is always open and measurable.
 \item If $\mathfrak{g}_\mathbb{R} = \ ^0\mathfrak{pq}(n)$ then:
  \begin{itemize}
   \item A flag domain $D$ is open if and only if $\delta$ is odd-symmetrizable.
   \item A flag domain $D$ is strongly measurable if and only if $\delta$ is odd-symmetric.
   \item If $\delta$ is $\Pi$-symmetric then $D$ is weakly measurable.
  \end{itemize}
 \item If $\mathfrak{g}_\mathbb{R} = \mathfrak{us\pi}(n)$, then $D$ is open if and only if $\delta$ is $\Pi$-symmetric. 
       If it is open it is always measurable.
\end{enumerate}

\end{satz}

Note that the second part of this theorem is immediate, as the involution in question is $\tau(\alpha) = -\alpha$.

\begin{rem}
The statements on strong measurability can be proven by analysis of the extended Dynkin diagram, but the point of view we adopted gives more information, in particular it allows to
prove statements on maximal odd dimension and weak measurability as well.
\end{rem}

\subsubsection*{Proof of Theorem}
 
If $\mathfrak{h}$ is a $\tau$-invariant Cartan subalgebra, then there is basis $e_1, \ldots, e_n,f_1, \ldots, f_m$ of $\mathbb{C}^{n \vert m}$, such that
all $e_i,f_j$ are common eigenvectors of $\mathfrak{h}$. Let

\[ x_i (\mathrm{diag}(\lambda_1,\ldots, \lambda_n, \mu_1, \ldots, \mu_n)) = \lambda_i \ \textnormal{and} \ 
y_j (\mathrm{diag}(\lambda_1,\ldots, \lambda_n, \mu_1, \ldots, \mu_n)) = \mu_j \]

\noindent Then $(y_j - x_i) \in \Phi \Leftrightarrow \exists X \in \mathfrak{p}: X(e_i) = f_j$ 
and analogous conditions hold for the even roots.

\subsubsection{The Case $\mathfrak{g}_\mathbb{R} = \mathfrak{sl}_{m \vert n}(\mathbb{R})$}

The open $G_{0\mathbb{R}}$-orbit in $G_0/B_0$, where $B_0 = B_n^+ \times B_m^+$ is the product of the usual Borel subalgebras of
$SL_n(\mathbb{C})$ and $SL_m(\mathbb{C})$, is open and projects onto the $G_{0\mathbb{R}}$-orbit in $Z_0$. So 

\[ (y_j - x_i) \in \Phi \Leftrightarrow \exists X \in \mathfrak{p}: X(e_i) = f_j \Leftrightarrow j \leq \min \{d_1 : d_0 \vert d_1 \in \delta, i \leq d_0\} \]  

\noindent As the intersection $\mathfrak{h} \cap \mathfrak{sl}_{n \vert m}(\mathbb{R})$ is not maximally compact, the open orbit does
not contain the neutral point of $G_0/P_0$. It is therefore useful to pass to the isomorphic real form
$\mathfrak{g}_\mathbb{R}^\prime = g \mathfrak{g}_\mathbb{R} g^{-1}$, where $g = c_1 \ldots c_{\lfloor \frac{n}{2} \rfloor} \tilde{c_1} \ldots \tilde{c}_{\lfloor \frac{m}{2} \rfloor}$ is
the product of commuting Cayley transforms in the subspaces $\langle e_i , e_{n-i+1} \rangle_\mathbb{C}$ and $\langle f_j , f_{m-j+1} \rangle_\mathbb{C}$.

The $\mathbb{C}$-antilinear involution defining $\mathfrak{g}_\mathbb{R}^\prime$ is $\tau: \mathfrak{g} \rightarrow \mathfrak{g}, X \mapsto \mathrm{Ad}(\mathrm{diag}(A_n,A_m))(\bar{X})$,
where $A_n$ and $A_m$ are the respective antidiagonal unit matrices.

Its action on $\Sigma(\mathfrak{g} : \mathfrak{h})$ is given by

\[ \tau(x_j - x_i) = x_{n-j+1} -x_{n-i+1}\]
\[ \tau(y_j - y_i) = y_{m-j+1} -y_{m-i+1}\]
\[ \tau(y_j - x_i) = y_{m-j+1} -x_{n-i+1}\]

\begin{rem}
 
Instead of considering a conjugate real form, one could consider a different base point $z$ in $G/B$, given by
the basis $e_1 + i e_n, e_2 + i e_{n-1} \ldots, e_2 - i e_{n-1}, e_1 - i e_n, f_1 + i f_m, \ldots, f_1 - i f_m$.
Then the stabilizer of $z$ is a parabolic subalgebra $\mathfrak{p}^\prime = g^{-1} \mathfrak{p} g$ of $\mathfrak{g}$
and $\mathfrak{g}_\mathbb{R} \cap g^{-1} \mathfrak{h} g$ is maximally compact and therefore $(G_\mathbb{R} \cdot z)_0$ 
is open. 

\end{rem}

\begin{prop}
 
A flag domain $D$ is open, if and only if $\delta$ is even-symmetrizable.

\end{prop}

\begin{proof}
 
First, suppose $D$ is not open. Then there is an odd root $\alpha \in \Phi^c \cap \tau \Phi^c$. Without loss of generality,
one may assume $\alpha = y_j - x_i$. Let $\overline{d_0} \vert \overline{d_1} = \min\{ d_0 \vert d_1 \in \delta : i \leq d_0 \}$
and $\tilde{d_0} \vert \tilde{d_1} = \min \{ d_0 \vert d_1 \in \delta : n-i+1 \leq d_0 \}$. As $(y_j - x_i) \not\in \Phi$, $j > \overline{d_1}$.
Moreover, as $\tau(y_j - x_i) = (y_{m-j+1} - x_{n-i+1}) \not\in \Phi$, also $m-j+1 > \tilde{d_1} \Leftrightarrow j \leq m-\tilde{d_1}$. 
Finally, as $n-i+1 \leq \tilde{d_0}$ implies $i > n-\tilde{d_0}$, one obtains the following inequalities:

\[ n-\tilde{d_0} < i \leq \overline{d_0}, \quad \overline{d_1} < j \leq m-\tilde{d_0} \]

\noindent So $\overline{d_0} \vert \overline{d_1}$ and $n-\tilde{d_0},m-\tilde{d_1}$ are not comparable and therefore $\delta$ is not symmetrizable.

Conversely, let $d_0 \vert d_1, d_0^\prime \vert d_1^\prime \in \delta$ such that $n-d_0^\prime < d_0$ and $m-d_1^\prime > d_1$.
Then the choice $j = m-d_1^\prime, i = d_0$ satisfies $(y_j - x_i) \in \Phi^c \cap \tau\Phi^c$ and thus $D$ is not open.
 
\end{proof}

\begin{prop}
 
A flag domain $D$ is strongly measurable, if and only if $\delta$ is even-symmetric.

\end{prop}

\begin{proof}
 
First assume $\delta$ is symmetric and let $\alpha = (y_j - x_i) \in \Phi$. We need to show $\tau(-\alpha) = (x_{n-i+1} - y_{m-j+1}) \in \Phi$.
Let $\underline{d_0} \vert \underline{d_1} = \max \{ d_0 \vert d_1 \in \delta: i > d_0 \}$ and $\overline{d_0} \vert \overline{d_1} = \min \{ d_0 \vert d_1 \in \delta: i \leq d_0 \}$.
Furthermore, let $\tilde{d_0} \vert \tilde{d_1} = \min \{ d_0 \vert d_1 \in \delta: m-j+1 \leq d_1\}$. Then $(x_{n-i+1} - y_{m-j+1}) \in \Phi$,
if and only if $n - i + 1 \leq \tilde{d_0}$. 

Now $m-j+1 > m - \overline{d_1}$ and by symmetry, $n - \underline{d_0} \vert m - \underline{d_1}$ is the succesor of $n - \overline{d_0} \vert m - \overline{d_1}$ in $\delta$,
so $\tilde{d_0} \geq n - \underline{d_0}$ and as $i > \underline{d_0}$, this implies $n - i + 1 \leq n - \underline{d_0} \leq \tilde{d_0}$, so $\tau(-\alpha) \in \Phi$.

Now assume $D$ is not strongly measurable, so there exists $\alpha = y_j - x_i \in \Phi$, such that $\tau(-\alpha) = x_{n-i+1} - y_{m-j+1} \not\in \Phi$.
Let $\underline{d_0} \vert \underline{d_1}, \overline{d_0} \vert \overline{d_1}$ and $\tilde{d_0} \vert \tilde{d_1}$ as before.
As $\tau(-\alpha) \not\in \Phi$, $n - i + 1 > \tilde{d_0}$. But $i > \underline{d_0}$, so $n- \overline{d_0} < n-i+1 \leq n-\underline{d_0}$. Therefore $\tilde{d_0} < n - \underline{d_0}$. 
Also $\tilde{d_1} \geq m-j+1 > m-j \geq m-\overline{d_1}$. Altogether, this yields

\[ n-\overline{d_0} \vert m-\overline{d_1} < \tilde{d_0} \vert \tilde{d_1} < n-\underline{d_0} \vert m-\underline{d_1} \]

\noindent And, as $n-\tilde{d_0} \vert m - \tilde{d_1} \not\in \delta$, $\delta$ is not symmetric. The proof proceeds analagously for even roots.

\end{proof}

\begin{prop}
 
If $n=m$ and $\delta$ is $Pi$-symmetric then $D$ is weakly measurable.

\end{prop}

\begin{proof}
 
Let $\alpha = (x_j - x_i) \in \Phi \cap \tau\Phi$, so $\tau(\alpha) = x_{n-j+1} - x_{n-i+1} \in \Phi$. $\Pi$-symmetry 
is equivalent to the fact, that if one of $x_j - x_i, y_j - y_i, x_j - y_i$ and $y_j - x_i$ is in $\Phi$, then so are the other three.
The same fact applied to $\tau(\alpha)$ yields, that they are all in $\Phi \cap \tau\Phi$(Note that this would fail, if $n \neq m$).
Conversely, if $\alpha$ is not in $\Phi \cap \tau\Phi$, then so neither will be the other three. Now, when computing the supertrace of the action of $\mathfrak{p}$ on $(\mathfrak{g}/\mathfrak{p})^*$, the contributions of these four
roots will add up to zero. As $\alpha$ was arbitrary, the supertrace will be zero and therefore $D$ is weakly measurable.

\end{proof}

\begin{prop}
 
Let $G_\mathbb{R} = PSL_{n \vert n}(\mathbb{R})$ or $G_\mathbb{R} = PSL_{k \vert k}(\mathbb{H})$ 
and $Z = \mathbb{P}(\mathbb{C}^{n \vert n})$. Then the open $G_\mathbb{R}$-orbits in $Z$ are weakly measurable.

\end{prop}

\begin{proof}
 
Suppose we are given a $\tau$-generic basis of $\mathbb{C}^n$. Then $\mathfrak{p}$ is the stabiliser of $e_1$ and $\tau \mathfrak{p}$ is
the stabiliser of $e_n$. Therefore

\[ \Sigma \setminus (\Phi \cap \tau\Phi) = \] 
\[ \{ x_j - x_1 : j > 1 \} \cup \{ x_j - x_n : j < n \} \cup \{ y_j - x_1 : 1 \leq j \leq n \} \cup \{ y_j - x_n : 1 \leq j \leq n \} \]

\noindent The graded sum of these roots is

\[ \sum_{j=2}^n (x_j - x_1) + \sum_{j=1}^{n-1} (x_j - x_n) - \sum_{j=1}^n (y_j - x_1) - \sum_{j=1}^n (y_j - x_n) \]

\[ = 2 \sum_{j=2}^{n-1} x_j - (n-2)(x_1 + x_n) - (2 \sum_{j=1}^n y_j - n(x_1 + x_n))\] 

\[ = 2 \sum_{j=1}^n x_j - 2 \sum_{j=1}^n y_j = 2 \mathrm{str} = 0    \]

\noindent Consequently $D$ is weakly measurable.

\end{proof}

\subsubsection{The Case $\mathfrak{g}_\mathbb{R} = \ ^0\mathfrak{pq}(n)$}

In this case $n =m$ and the defining involution of $\mathfrak{g}_\mathbb{R}$ is 

\[ \tau \begin{pmatrix} A & B \\ C & D \end{pmatrix} = \begin{pmatrix} \bar{D} & \bar{C} \\ \bar{B} & \bar{A} \end{pmatrix} \]

\noindent and the action of $ \tau $ on $\Sigma(\mathfrak{g} : \mathfrak{h})$ is given by

\[ \tau(x_j - x_i) = y_j - x_i \]
\[ \tau(y_j - x_i) = x_j - y_i \]

\noindent Moreover the $G_{0\mathbb{R}}$-orbit through the neutral point in $G_0/B_0$  is open, where $B_0 = B_n^+ \times B_n^-$ is the product of the usual Borel subalgebra of
$SL_n(\mathbb{C})$ and its opposite. It projects onto the open $G_{0\mathbb{R}}$-orbit in $Z_0$. So 

\[ (y_j - x_i) \in \Phi \Leftrightarrow \exists X \in \mathfrak{p}: X(e_i) = f_j \Leftrightarrow n-j+1 \leq \min \{d_1 : d_0 \vert d_1 \in \delta, i \leq d_0\} \]  

\begin{prop}
 
A flag domain $D$ is open if and only if $\delta$ is odd-symmetrizable.

\end{prop}

\begin{proof}
 
First suppose $D$ is not open. Then there is an odd root $\alpha \in \Phi^c \cap \tau \Phi^c$. Without loss of generality,
one may assume $\alpha = y_j - x_i$. Let $\overline{d_0} \vert \overline{d_1} = \min\{ d_0 \vert d_1 \in \delta : i \leq d_0 \}$
and $\tilde{d_0} \vert \tilde{d_1} = \min \{ d_0 \vert d_1 \in \delta : n-i+1 \leq d_1 \}$. As $(y_j - x_i) \not\in \Phi$, $n-j+1 > \overline{d_1}$.
Moreover as $\tau(y_j - x_i) = (x_j - y_i) \not\in \Phi$, also $j > \tilde{d_0}$. 
Finally as $n-i+1 \leq \tilde{d_1}$ implies $i > n-\tilde{d_1}$, one obtains the following inequalities:

\[ n-\tilde{d_1} < i \leq \overline{d_0}, \quad \tilde{d_0} < j \leq n-\overline{d_1} \]

\noindent So $\overline{d_0} \vert \overline{d_1}$ and $\tilde{d_0} \vert \tilde{d_1}$ are not comparable and therefore $\delta$ is not symmetrizable.

Conversely, let $d_0 \vert d_1, d_0^\prime \vert d_1^\prime \in \delta$ such that $n-d_1^\prime < d_0$ and $n-d_1 > d_0^\prime$.
Then the choice $j = n - d_1, i = d_0^\prime$ satisfies $(y_j - x_i) \in \Phi^c \cap \tau\Phi^c$ and thus $D$ is not open.
 
\end{proof}

\begin{prop}
 
A flag domain $D$ is strongly measurable, if and only if $\delta$ is odd-symmetric.

\end{prop}

\begin{proof}
 
First assume $\delta$ is symmetric and let $\alpha = (y_j - x_i) \in \Phi$. We need to show $\tau(-\alpha) = (y_i - x_j) \in \Phi$.
Let $\underline{d_0} \vert \underline{d_1} = \max \{ d_0 \vert d_1 \in \delta: i > d_0 \}$ and $\overline{d_0} \vert \overline{d_1} = \min \{ d_0 \vert d_1 \in \delta: i \leq d_0 \}$.
Furthermore let $\tilde{d_0} \vert \tilde{d_1} = \min \{ d_0 \vert d_1 \in \delta: j \leq d_0\}$. Then $(y_i - x_j) \in \Phi$,
if and only if $ n-i+1 \leq \tilde{d_1}$. 

Now $n - \overline{d_1} < j$ and by symmetry, $n - \underline{d_1} \vert n - \underline{d_0}$ is the succesor of $n - \overline{d_1} \vert n - \overline{d_0}$ in $\delta$,
so $\tilde{d_0} \geq n - \underline{d_1}$ and as $i > \underline{d_0}$, this implies $n - i + 1 \leq n - \underline{d_0} \leq \tilde{d_1}$, so $\tau(-\alpha) \in \Phi$.

Now assume $D$ is not strongly measurable, so there exists $\alpha = y_j - x_i \in \Phi$, such that $\tau(-\alpha) = y_i - x-j \not\in \Phi$.
Let $\underline{d_0} \vert \underline{d_1}, \overline{d_0} \vert \overline{d_1}$ and $\tilde{d_0} \vert \tilde{d_1}$ as before.
As $\tau(-\alpha) \not\in \Phi$, $n - i + 1 > \tilde{d_1}$. But $i > \underline{d_0}$, so $n- \overline{d_0} < n-i+1 \leq n-\underline{d_0}$. Therefore $\tilde{d_1} < n - \underline{d_0}$. 
Also $\tilde{d_0} \geq j >  n-\overline{d_1}$. Altogether this yields

\[ n-\overline{d_1} \vert n-\overline{d_0} < \tilde{d_0} \vert \tilde{d_1} < n-\underline{d_1} \vert n-\underline{d_0} \]

\noindent And as $n-\tilde{d_1} \vert n - \tilde{d_0} \not\in \delta$, $\delta$ is not symmetric. The proof proceeds analagously for even roots.

\end{proof}

\begin{prop}
 
If $\delta$ is $\Pi$-symmetric then $D$ is weakly measurable.

\end{prop}

\begin{proof}
 
Let $\alpha = (x_j - x_i) \in \Phi \cap \tau\Phi$, so $\tau(\alpha) = y_j - x_i \in \Phi$. $\Pi$-symmetry 
is equivalent to the fact that if one of $x_j - x_i, y_{n-j+1} - y_{n-i+1}, x_j - y_{n-i+1}$ and $y_{n-j+1} - x_i$ is in $\Phi$, then so are the other three.
The same fact applied to $\tau(\alpha)$ yields that they are all in $\Phi \cap \tau\Phi$.
Conversely if $\alpha$ is not in $\Phi \cap \tau\Phi$ then so neither will be the other three. 
Now when computing the supertrace of the action of $\mathfrak{p}$ on $(\mathfrak{g}/\mathfrak{p})^*$ the contributions of these four
roots will add up to zero. As $\alpha$ was arbitrary, the supertrace will be zero and therefore $D$ is weakly measurable.

\end{proof}

\subsubsection{The Case $\mathfrak{g}_\mathbb{R} = \mathfrak{us}\pi(n)$}

In this case $n = m$ and the defining involution is

\[ \tau \begin{pmatrix} A & B \\ C & D \end{pmatrix} = \begin{pmatrix} -D^\dagger & B^\dagger \\ -C^\dagger & -A^\dagger \end{pmatrix} \]

\noindent and the action of $\tau$ on $\Sigma(\mathfrak{g} : \mathfrak{h})$ is

\[ \tau(x_j - x_i) = y_i - y_j \]
\[ \tau(y_j - x_i) = y_i - x_j \]  

\noindent Here the $G_{0\mathbb{R}}$-orbit through the neutral point in $G_0/B_0$, where $B_0 = B_n^+ \times B_n^+$, is open and
projects onto $D_0 \subseteq Z_0$.

\begin{prop}
 
A flag domain $D$ is open if and only if $\delta$ is $\Pi$-symmetric. If this is the case $D$ is strongly measurable.

\end{prop}

\begin{proof}
 
The involution $\tau$ acts trivially on the roots $\pm (y_i - x_i)$ for all $1 \leq i \leq n$. So for $\Phi \cap \tau \Phi$
to be empty, one needs $\pm (y_i - x_i) \in \Phi$ for all $1 \leq i \leq n$. This is equivalent to $\Pi$-symmetry.

$\Pi$-symmetry also implies that if one of $x_j - x_i, y_j - y_i, y_j - x_i$ and $x_j - y_i$ is in $\Phi$, then
so are the other three. Now $y_j - y_i = \tau(-(x_j - x_i))$ and $y_j - x_i = \tau(-(y_j - x_i))$. So $\Pi$-symmetry
also yields strong measurability. 

\end{proof}

\subsection{Types B(n,m),C(m) and D(n,m), orthosymplectic superalgebras} 

Now suppose we are given we are given a $\mathbb{C}$-supervector space $V^{k \vert 2m}$ with a non-degenerate even super-symmetric
bilinear form $S: V \times V \rightarrow \mathbb{C}$. If $\mathfrak{g}$ is of type $B(n,m)$, $k = 2n+1$, if it is of type $C(m)$, then $k = 2$
and if it is of type $D(n,m)$, then $k = 2n$. One can always choose a basis $e_1, \ldots, e_k,f_1, \ldots, f_{2m}$ of $V$ such that

\[ S(e_i,f_l) = 0 , S(e_i,e_j) = \delta_{i,k-j}, S(f_l,f_a) = \delta_{l,2m-a} \forall 1 \leq i,j \leq k, 1 \leq l,a \leq 2m \]  

\noindent Moreover if $\mathfrak{g}$ stabilizes a subspace $W \subseteq V$, then it also stabilizes the orthogonal complement $W^\perp$.
Consequently every dimension sequence will automatically be even-symmetric. For those real forms where $\tau$
acts on $\Sigma$ by $\tau(\alpha) = - \alpha$, every orbit $D$ with $D_0$ open will therefore always have maximal odd dimension 
and be measurable. The only real form for which this is not the case is the real form $\mathfrak{g}_\mathbb{R} = \mathfrak{osp}(2p+1,2q+1 \vert 2m)$
of $\mathfrak{g} = \mathfrak{osp}(2n \vert 2m)$.

In that particular case the action of $\tau$ on $\Sigma$ is given by

\[ \tau(x_n - x_j) = x_n + x_j , \tau(x_n - y_j) = x_n + y_j , \tau(\alpha) = - \alpha, \ \textnormal{else} \]

\noindent This yields the following

\begin{satz}

Let $\mathfrak{g} = \mathfrak{osp}(2n \vert 2m)$, $\mathfrak{g}_\mathbb{R} = \mathfrak{osp}(2p+1,2q+1 \vert 2m)$ and $D$ as above. Then:

\begin{enumerate}
 \item A flag domain $D$ has maximal odd dimension, if and only if $n \vert d \not\in \delta$ for all $d < m$
 \item A flag domain $D$ is strongly measurable, if and only if $n \vert m \not\in \delta$ or $n-1 \vert m \in \delta$
 \item A flag domain $D$ is weakly measurable, if and only if $n \vert m \in \delta$ and its immediate predecessor is $n-d-1 \vert m-d$, $1 \leq d \leq \min \{n-1, m\}$ 
\end{enumerate}
  
\end{satz}

\begin{proof}
 
1. If $\dim_1(D)$ is not maximal there must be some $\alpha \in \Phi^c \cap \tau\Phi^c$. By virtue of even symmetry
this can only happen for those roots which satisfy $\tau(\alpha) \neq -\alpha$. As $-x_n - y_j \in \Phi$ for all $1 \leq j \leq m$, 
we may assume $\alpha = x_n - y_j$. Then $\tau(\alpha) = x_n + y_j$, which is always an element of $\Phi^c$ because of even symmetry.
So $D$ does not have maximal odd dimension precisely when $x_n - y_j \in \Phi^c$ for some $1 \leq j \leq m$. This is only the case, if $n \vert d \in \delta$ for some $d < j$.

2. $D_0$ is measurable if and only if either $n \not\in \delta_0$ or $n-1 \in \delta_0$. If $n \vert m \not\in \delta$ or $n-1 \vert m \in \delta$, then $x_n \pm y_j \in \Phi$,
if and only if $- x_n \pm y_j \in \Phi$, so these contributions cancel in the sum of odd roots, yielding strong measurability.

3. Assume $n-d-1 \vert m-d \in \delta$ and $n \vert m$ is its succesor in $\delta$. Then the even roots $\alpha \in \Phi \cap \tau\Phi$
with $-\alpha \not\in \Phi \cap \tau\Phi$ are $-x_j - x_n$ and $x_j - x_n$ for all $n-d-1 \leq j \leq n-1$. On the other hand 
the odd roots $\alpha \in \Phi \cap \tau\Phi$ with $-\alpha \not\in \Phi \cap \tau\Phi$ are $-x_n - y_j$ and $-x_n + y_j$ for all $m - d \leq j \leq m$.
The graded sum of all these roots is zero so $D$ is weakly measurable. The converse stems from the fact
that the odd roots will always be the given ones and the only way to cancel them with even roots is with the given even roots. 

\end{proof}

This completes the characterization of measurability for the orthosymplectic superalgebras.

\subsection{The exceptional Lie superalgebras}

As the existence of non-measurable open orbits requires the existence of a non-trivial automorphism of the Dynkin diagram,
these can only occur for the exceptional superalgebras $E_6$ and $D(2,1,\alpha)$. For all other exceptional Lie superalgebras and all of their real forms,
the open orbits in any $Z(\delta) = G/P$ will be measurable. Now consider the two remaining cases:

The exceptional family $D(2,1,\alpha)$ does have the same root system as the simple Lie superalgebra $D(2,1)$. Moreover,
if $\mathfrak{g}$ is of type $D(2,1,\alpha)$, then $\mathfrak{g}_0 \cong \mathfrak{sl}_2(\mathbb{C}) \oplus \mathfrak{sl}_2(\mathbb{C}) \oplus \mathfrak{sl}_2(\mathbb{C})$ 
and there are three real forms with respective even parts $\mathfrak{g}_{0\mathbb{R}} \cong \mathfrak{sl}_2(\mathbb{R}) \oplus \mathfrak{sl}_2(\mathbb{R}) \oplus \mathfrak{sl}_2(\mathbb{R})$, 
$\mathfrak{g}_{0\mathbb{R}} \cong \mathfrak{sl}_2(\mathbb{R}) \oplus \mathfrak{su}(2) \oplus \mathfrak{su}(2)$ or $\mathfrak{g}_{0\mathbb{R}} \cong \mathfrak{sl}_2(\mathbb{R}) \oplus \mathfrak{sl}_2(\mathbb{C})$.

For the first two of these, the action of $\tau$ on the root system is $\tau(\alpha) = -\alpha$, so $D$ will always have maximal odd dimenson and be measurable.
For the third real form, the action of $\tau$ on the root system is given by

 \[\tau(x_1 - x_2) = (-x_1 - x_2), \tau(\pm y \pm x_1) = \mp y \mp x_1, \tau(\pm y \pm x_2) = \mp y \pm x_2\]

\noindent This is the same action as for the real form $\mathfrak{osp}(1,3 \vert 2)$ of $\mathfrak{osp}(4,2)$. 
Moreover, this real form can only occur, if $\alpha = 1, -\frac{1}{2}$ or $-2$ and then $\mathfrak{g}$ and $\mathfrak{g}_\mathbb{R}$ are 
isomorphic to $\mathfrak{osp}(4,2)$ and $\mathfrak{osp}(1,3 \vert 2)$ respectively. The measurable open orbits in this case have been
classified above.


For $E_6$ there are two real forms allowing non-measurable open orbits, namely $E_{6,C_4}$ and $E_{6,F_4}$. 
Non-measurable open orbits for these real forms can be constructed by the procedure outlined above.

\subsection{Type P(n)}

Now consider the supervectorspace $V = \mathbb{C}^{n \vert n}$ with a non-degenerate odd super-skewsymmetric bilinear form $\omega$.
Then there is a standard basis $e_1, \ldots, e_n,$  $f_1, \ldots, f_n$ of $V$ such that

\[ \omega(e_i,e_j) = \omega(f_i,f_j) = 0 , \omega(e_i,f_j) = \delta_{ij} \forall 1 \leq i,j \leq n \]

\noindent If $\mathfrak{g} = \mathfrak{\pi}(n)$ stabilizes a subspace $W \subseteq V$, then it also stabilizes $W^\perp$, so every dimension
sequence will automatically be odd-symmetric. The only possible real forms of $\mathfrak{g}$ are $\mathfrak{g}_\mathbb{R} = \pi_\mathbb{R}(n)$
and $\pi_\mathbb{H}(m)$, if $n = 2m$ is even. In both cases the action of $\tau$ on $\Sigma$ is the following:

\[ \tau(\pm x_i \pm x_j) = \pm x_{n-i+1} \pm x_{n-j+1} \]

\noindent Note especially that the roots $\alpha = -2x_j$, for which $-\alpha \not\in \Sigma$ are fixed by $\tau$. This implies the following:

\begin{satz}
 
A flag domain $D$ has maximal odd dimension if and only if $\delta$ is $\Pi$-symmetric.

\end{satz}
 
\begin{proof}
 
For $\dim_1(D)$ to be maximal all roots fixed by $\tau$ must be in $\Phi$. These are presiely the roots $\alpha = \pm (x_i + x_{n-i+1})$.
Now $ - x_i - x_{n-i+1} \in \Phi$ if and only if there is some $X \in \mathfrak{p}$ such that $X(f_i) = e_{n-i+1}$. For this to hold for all 
$1 \leq i \leq n$ one needs $d_0 \geq d_1$ for all $d \in \delta$. On the other hand, $- x_i - x_{n-i+1} \in \Phi$ requires $d_1 \geq d_0$
for all $d \in \delta$. As both these conditions must be satisfied, $\delta$ is $\Pi$-symmetric. 

Conversely suppose there is an odd root $\alpha \in \Phi^c \cap \tau\Phi^c$, without loss of generality $\alpha = - x_i - x_j$. 
Then $n - j + 1 > \overline{d_0}$, where $\overline{d} = \min \{ d \delta : d_1 \geq i \}$. Moreover, $\tau(\alpha) = - x_{n-i+1} - x_{n-j+1} \not\in \Phi$,
so $ j > \tilde{d_0} $, where $\tilde{d} = \min \{ d \in \delta : d_1 \geq n - i + 1 \}$. This yields:

\[ \tilde{d_0} < j \leq n - \overline{d_0}  , n - \tilde{d_1} < i \leq \overline{d_1}  \]

\noindent Now suppose $\delta$ is $\Pi$-symmetric. Then $\overline{d_0} = \overline{d_1}$ and $\tilde{d_0} = \tilde{d_1}$
and consequently $n = \tilde{d_0} + n - \tilde{d_0} < n - \overline{d_0} + \overline{d_0} = n$, which is a contradiction.   

\end{proof}

Note that given $\Pi$-symmetry the even and odd symmetry conditions coincide. Consequently if $D$ has maximal odd dimension
then $D_0$ is automatically measurable. Also for all odd roots $\alpha$ one has $\alpha \in \Phi$ if and only if $-\alpha \in \Phi$.
So for $D$ to be strongly measurable one needs to consider the roots $2x_i$ whose negatives are no roots.

\begin{satz}
 
A flag domain $D$ is strongly measurable if and only if $n = 2m$ is even and $m \vert m \in \delta$.

\end{satz}
 
\begin{proof}

If $n = 2m+1$ is odd then $\alpha = 2x_{m + 1} \in \Phi \cap \tau\Phi$ in any case(by virtue of the condition for maximal odd dimension), 
so $D$ cannot be measurable.

Now let $n = 2m$ be even. $2x_i$ is an element of $\Phi$, if and only if there is some $X \in \mathfrak{p}$ such that $X(e_i) = f_i$. This requires
$n - i + 1 \leq \overline{d}$, where $\overline{d} \vert \overline{d} = \min \{ d \vert d \in \delta: i \leq d \}$. 
This implies $n + 1 \leq 2\overline{d}$. Analogously $2x_i \in \tau\Phi$ if and only if $i \leq \tilde{d}$, 
where $\tilde{d} \vert \tilde{d} = \min \{ d \vert d \in \delta: n - i + 1 \leq d \}$. By the odd symmetry condition $\tilde{d} = n - \underline{d} + 1$,
where $\underline{d} \vert \underline{d} = \max \{ d \vert d \in \delta : d < i \}$. Summarizing this, one has $2x_i \in \Phi \cap \tau\Phi$ if and only if

\[ 2 (\underline{d} + 1) \leq n + 1 \leq 2 \overline{d} \]

\noindent Now if $m \vert m \not\in \delta$, then $2x_m \in \Phi \cap \tau\Phi$, as in that case $\underline{d} < m$ and $\overline{d} > m$, 
so the above condition is satisfied. Conversely if $m \vert m \in \delta$, then $\underline{d} + 1 $ and $\overline{d}$ are always
either both less or equal to $m$ or both greater than $m$. Consequently $2x_i \not\in \Phi \cap \tau\Phi$ for all $1 \leq i \leq n$ and
therefore $D$ is measurable.  
 
\end{proof}

\subsection{Type Q(n)}

  
The Lie superalgebras of type $Q$ have the important trait that all root spaces are $1 \vert 1$-dimensional and therefore 
for any parabolic subalgebra $\mathfrak{p} \subseteq \mathfrak{g}$ one always has $\dim_1(\mathfrak{p}) = \dim_0(\mathfrak{p})$.
Therefore if $D_0$ is open, then $D$ will automatically have maximal odd dimension. Moreover as the root spaces are $1 \vert 1$- dimensional,
every root is even and odd and a graded sum of roots is therefore always zero. This yields the following classification result:

\begin{satz}
 
A flag domain $D$ always has maximal odd dimension and is always weakly measurable. It is strongly measurable if and only if $D_0$ is measurable.

\end{satz}

\section{Even real forms and measurable flag domain}
\subsection{Definition of even real forms and characterization of strong measurability}
As the Berezinian module of a flag domain $D = G_\mathbb{R}/L_\mathbb{R}$ is $1$-dimensional the action of $\mathfrak{l}_{1\mathbb{R}}$
on it is necessarily trivial. It is therefore reasonable to consider instead real forms of the even group $G_0$ 
rather than real forms of the whole supergroup $G$. However, as the action of the $\mathbb{C}$-antilinear involution $\tau$
is central to the classification, one should actually consider a real form $G_{0\mathbb{R}}$ defined by an isomorphism of $\mathfrak{g}$.
This leads to the following definition:

\begin{def1}
\begin{enumerate}
 \item Let $\mathfrak{g}$ be a complex Lie superalgebra. A real Lie algebra $\mathfrak{g}_{0\mathbb{R}}$ is an even real form of $\mathfrak{g}$ 
if and only if $\mathfrak{g}_{0} = \mathrm{Fix}(\tau)$ for a $\mathbb{C}$-antilinear isomorphism $\tau: \mathfrak{g} \rightarrow \mathfrak{g}$
satisfying $\tau^2(X) = (-1)^{\vert X \vert} X$ for homogeneous $X \in \mathfrak{g}$.
 
\item Let $Z = G/P$ be a flag supermanifold. A $G_{0\mathbb{R}}$-flag domain in $Z$ is an open submanifold $D \subseteq Z$ such that
$D_0$ is an open $G_{0\mathbb{R}}$-orbit in $Z_0$. It is measurable if it allows a $G_{0\mathbb{R}}$-invariant Berezinian density. 
\end{enumerate}
\end{def1}

As before there is a distinction between strong and weak measurability and a characterization theorem for strong measurability:

\begin{satz}
 
Let $Z = G/P$ be a flag supermanifold, $G_{0\mathbb{R}} = \mathrm{Fix}(\tau)$ an even real form and $D \subseteq Z$ a flag domain.
Then the following are equivalent:

\begin{enumerate}
 \item $D$ is strongly measurable
 \item $\mathfrak{p} \cap \tau\mathfrak{p}$ is reductive
 \item $\mathfrak{p} \cap \tau\mathfrak{p}$ is the reductive part of $\mathfrak{p}$
 \item $\tau \Phi^r = \Phi^r$ and $\tau \Phi^n = \Phi^c$
 \item $\tau\mathfrak{p}$ is $G_0$-conjugate to $\mathfrak{p}^{op}$.
\end{enumerate}

\end{satz}

\begin{proof}
 
Let $z_0 \in D_0$ be a base point and $\mathfrak{p} = \mathrm{Stab}_\mathfrak{g}(z_0)$. As $D_0$ is an open $G_{0\mathbb{R}}$-orbit
$\mathfrak{p} \cap \mathfrak{g}_{0\mathbb{R}}$ contains a maximal compact Cartan subalgebra $\mathfrak{h}_\mathbb{R}$ of $\mathfrak{g}_{0\mathbb{R}}$.
The adjoint representation of $G$ induces a representation $\tilde{\mathrm{ad}}$ of $\mathfrak{h}_\mathbb{R}$ on $(\mathfrak{g}/\mathfrak{p})^\mathbb{R}$
Let $H \in \mathfrak{h}_\mathbb{R}$. Then

\[ \mathrm{str}(\tilde{\mathrm{ad}}(H)) = 2 \sum_{\alpha \in \Phi_0^c} \mathrm{Re} \ \alpha(H) - 2 \sum_{\alpha \in \Phi_1^c} \mathrm{Re} \ \alpha(H)\]
\[ = \sum_{\alpha \in \Phi_0^c} (\alpha(H) + \overline{\alpha(H)}) - \sum_{\alpha \in \Phi_1^c} (\alpha(H) + \overline{\alpha(H)})\]
\[ = \sum_{\alpha \in \Phi_0^c} (\alpha(H) + \tau \cdot \alpha(H)) - \sum_{\alpha \in \Phi_1^c} (\alpha(H) + \tau \cdot \alpha(H))\]
\[ = \sum_{\alpha \in \Phi_0^c} \alpha(H) + \sum_{\alpha \in \tau\Phi_0^c} \alpha(H) - \sum_{\alpha \in \Phi_1^c} \alpha(H) - \sum_{\alpha \in \tau\Phi_1^c} \alpha(H)\]

Now assume $D_0$ is measurable. Then the sums over the even roots are zero and one obtains

\[ \mathrm{str}(\tilde{\mathrm{ad}}(H)) = - \sum_{\alpha \in \Phi_1^c} \alpha(H) - \sum_{\alpha \in \tau\Phi_1^c} \alpha(H) = 0 \]
if and only if $\tau\Phi_1^c = -\Phi_1^c$. So 1. and 4. are equivalent and equivalence of 2. - 5. follows again from the decomposition
of $\mathfrak{p}$ into its reductive and nilpotent parts. 

\end{proof}

\subsection{Classification of even real forms and measurability}

A classification of the isomorphisms $\tau$ defining even real forms is given in \cite{Ser}. They are the following:

Let $\mathfrak{g} = (\mathfrak{p})\mathfrak{sl}_{n \vert m}(\mathbb{C})$. Then there are the following real forms:

\begin{enumerate}
 \item The unitary groups $G_{0\mathbb{R}} =SU(p,n-p) \times SU(q, m-q) \times U(1)$. 
In this case the action on $\Sigma$ is given by $\tau(\alpha) = - \alpha$
for all $\alpha \in \Sigma$.
 \item If $m = 2k$ is even there is an even real form with $G_{0\mathbb{R}} = SL_n(\mathbb{R}) \times SL_k(\mathbb{H}) \times \mathbb{R}^{>0}$.
The action on the root system is given by
$\tau(x_j - x_i) = x_{n-j+1} - x_{n-i+1}, \tau(y_j - y_i) = y_{m-j+1} - y_{m-i+1}, \tau(y_j - x_i) = y_{m-j+1} - x_{n-i+1}$.
 \item If $n = m$ there is an even real form with $G_{0\mathbb{R}} = SL_n(\mathbb{C})$.
In that case the action on the root system is $\tau(x_j - x_i) = y_j - y_i, \tau(y_j-x_i) = x_j - y_i$. 
\end{enumerate}

Note that all given actions on the root system were already realized by real forms and that the action of $US\pi(n)$ on the root system
is not realized by an even real form.

Now suppose $\mathfrak{g} = \mathfrak{osp}_{n \vert 2m}(\mathbb{C})$. Then there are the even real forms $G_{0\mathbb{R}} = SO(p,q) \times Sp(2r,2s)$.
Unless $p$ and $q$ are both odd, the action on the root system is trivial and if $p$ and $q$ are both odd, the action on the root system
is identical to that of the real form $\mathrm{Osp}(p,q \vert 2m)$. If $n = 2k$ is even then there is another even real form which has
$G_{0\mathbb{R}} = SO^*(2k) \times Sp_\mathbb{R}(2m)$ and its action on the root system is trivial. Here again, the results for the even
real forms of $D(2,1,\alpha)$ are identical to those for $D(2,1)$. Contrary to the case of ordinary real forms, the simple Lie superalgebras
of type $P$ and $Q$ are excluded from this list as they do not allow even real forms. 

As the actions of the even real forms on the root systems have already been realized by real forms 
the conditions for weak and strong measurability are the same as for the real forms in question.
The correspondence is the following:

\noindent \begin{tabular}{ p{6cm} | p{6cm} }
 
 real form & even real form  \\ \hline
  $SU(p,n-p \vert q,m-q)$ & $SU(p,n-p) \times SU(q,m-q) \times U(1)$ \\ \hline
  $SL_{n \vert m}(\mathbb{R}),SL_{k \vert l}(\mathbb{H})$ & $SL_{n}(\mathbb{R}) \times SL_l(\mathbb{H}) \times \mathbb{R}^{>0}$ \\ \hline
  $ ^0PQ(n)$ & $SL_{n}(\mathbb{C})$ \\ \hline
  $US\pi(n)$  & none \\ \hline
  $\mathrm{Osp}(2p+1,2q+1 \vert 2m)$  & $SO(2p+1,2q+1) \times Sp(2m)$ \\ \hline

\end{tabular}
\section{Tables}

The three symmetry conditions are denoted by $\ev$, $\odd$ and three $\Pi$ respectively. For the even and odd symmetry condition,
the respective symmetrizability conditions are denoted by $\ev^*$ and $\odd^*$.

\begin{rem}
Note that the table below includes all information on the classical case as well:
The simple Lie algebras $A_{n-1}$ are realized by $\mathfrak{sl}_n(\mathbb{C})$ and the criteria for
measurability are analogous to the ones for the real forms of $A(n-1,m-1)$ given below. Moreover, there are
 identifications $B_n \cong B(n,0), D_n \cong D(n,0)$ and $C_m \cong D(0,m)$ and the exceptional simple
Lie algebras are explicitly contained in the table. 
\end{rem}

\medskip

\noindent \begin{tabular}{p{1.3cm} | p{3.7cm} | p{2.3cm} | p{2.5cm} | p{2cm}}
 
Type & real form & maximal odd dimension & weak measurability  & strong measurability \\ \hline
A(n,m) & $\mathfrak{sl}_{n+1 \vert m+1}(\mathbb{R}), \mathfrak{sl}_{k \vert l}(\mathbb{H})$ & $\ev^*$ & - & $\ev$ \\ \hline
       & $\mathfrak{su}(p, n-p \vert q, m-q)$ & always & - & always \\ \hline
A(n,n) & $\mathfrak{psl}_{n+1 \vert n+1}(\mathbb{R}), \mathfrak{psl}_{k \vert k}(\mathbb{H})$ & $\ev^*$ & $\Pi$ or $\mathbb{P}(\mathbb{C}^{n \vert n})$ & $\ev$ \\ \hline
       & $^0\mathfrak{pq}(n+1)$ & $\odd^*$ & $\Pi$ & $\odd$ \\ \hline
       & $\mathfrak{us\pi}(n+1)$ & $\Pi$ & - & $\Pi$ \\ \hline
B(n,m) & $\mathfrak{osp}(p,q \vert 2m)$ & always & - & always \\ \hline
C(m)   & $\mathfrak{osp}(2 \vert 2r, 2s)$ & always & - & always \\ \hline
D(n,m) & $\mathfrak{osp}(2p,2q \vert 2m)$ & always & - & always \\ \hline
       & $\mathfrak{osp}^*(2n \vert 2r, 2s)$ & always & - & always \\ \hline
       & $\mathfrak{osp}(2p+1,2q+1 \vert 2m)$ & $n \vert d \not\in \delta, d < m$ & $ \exists  1 \leq d \leq  \min \{ n - 1, m \}: (n - d - 1 \vert m - d < n \vert m) \subseteq \delta $ & $n \vert m \not\in \delta$ or \newline $n-1 \vert m \in \delta$ \\ \hline 
D(2,1,$\alpha$) & any & as for D(2,1) \\ \hline
$E_6$ & $E_{6,C_4},E_{6,F_4}$ & - & - & $\sigma(J) = J$ \\ \hline
$E_6$ & others & - & - & always \\ \hline
$E_8,E_7$ & any & - & - & always \\ \hline
$F_4,F(4)$ & any & always & - & always \\ \hline
$G_2,G(3)$ & any & always & - & always \\ \hline
P(n)   & $\mathfrak{s\pi}_\mathbb{R}(n), \mathfrak{s\pi}^*(n)$ & $\Pi$ & - & $n = 2k$ and \newline $k \vert k \in \delta$ \\ \hline
Q(n)   & $\mathfrak{pq}_\mathbb{R}(n), \mathfrak{pq}_\mathbb{H}(k), \mathfrak{upq}(p,q)$ & always & always & as for $A_n$\\ \hline

\end{tabular}
\newpage
The respective table for the even real forms is the following:

\noindent \begin{tabular}{p{1.3cm} | p{4cm}  | p{4cm} | p{4cm}}
 
Type & even real form & weak measurability  & strong measurability \\ \hline
A(n,m) & $\mathfrak{sl}_{n+1}(\mathbb{R}) \oplus \mathfrak{sl}_k(\mathbb{H}) \oplus \mathbb{R} $ & - & $\ev$ \\ \hline
       & $\mathfrak{su}(p, n-p) \oplus \mathfrak{su}(q, m-q) \oplus \mathfrak{u}(1)$ & - & always \\ \hline
A(n,n) & $\mathfrak{sl}_{n+1}(\mathbb{R}) \oplus \mathfrak{sl}_{k}(\mathbb{H})$ & $\Pi$ or $\mathbb{P}(\mathbb{C}^{n \vert n})$ & $\ev$ \\ \hline
       & $\mathfrak{sl}_{n+1}(\mathbb{C})$ & $\Pi$ & $\odd$ \\ \hline
B(n,m) & $\mathfrak{so}(p,q) \oplus \mathfrak{sp}(2r,2s)$ & - & always \\ \hline
C(m)   & $\mathfrak{so}(2) \oplus \mathfrak{sp}(2r, 2s)$ & - & always \\ \hline
D(n,m) & $\mathfrak{so}(2p,2q) \oplus \mathfrak{sp}(2r,2s)$ & - & always \\ \hline
       & $\mathfrak{so}^*(2n) \times \mathfrak{sp}_\mathbb{R}(2m)$ & - & always \\ \hline
       & $\mathfrak{so}(2p+1,2q+1) \oplus \mathfrak{sp}(2r,2s)$ & $ \exists  1 \leq d \leq  \min \{ n - 1, m \}: (n - d - 1 \vert m - d < n \vert m) \subseteq \delta $ & $n \vert m \not\in \delta$ or \newline $n-1 \vert m \in \delta$ \\ \hline 
D(2,1,$\alpha$) & any & as for D(2,1) \\ \hline

\end{tabular}

\begin{rem}

\item This paper will be part of the author's dissertation at the Ruhr-Universit\"{a}t Bochum.

\end{rem}

\end{document}